\documentclass[12pt]{amsart}
\usepackage{amssymb,hyperref}
\usepackage[mathscr]{eucal}
\usepackage[utf8]{inputenc}
\usepackage[T1]{fontenc}

\let\phi=\varphi
\newcommand{\phif}{\phi\!\cdot\!\!f}
\newcommand{\aff}{\mathop{\mbox{\rm aff}}}

\newcommand{\R}{{\mathbb R}}
\newcommand{\E}{{\mathscr E}}
\newcommand{\eq}[1]{\eqref{#1}}
\newcommand{\Eq}[2]{\begin{equation}\label{#1}#2\end{equation}}

\newtheorem{Lem}{Lemma}
\newtheorem{Thm}{Theorem}
\newtheorem*{CThm}{Ceva{}'s Theorem}
\newtheorem*{Cor}{Corollary}

\begin{document}
\date{\today}

\title[Characterization of segment and convexity preserving maps]
{Characterization of segment and convexity preserving maps}

\author[Zs. P\'ales]{Zsolt P\'ales}

\address{Institute of Mathematics, University of Debrecen, 
H-4010 Debrecen, Pf.\ 12, Hungary}

\email{pales@science.unideb.hu}



\thanks{This research has been supported by the Hungarian
Scientific Research Fund (OTKA) Grant NK81402.}


\begin{abstract}
In this note functions that transform open segments of a linear space
into open segments of another linear space are studied and characterized.
Assuming that the range is non-collinear, it is proved that
such a map can always be expressed as the ratio of two affine functions.
\end{abstract}

\maketitle

\section{Introduction}

Throughout this paper assume that $X$ and $Y$ are (real) linear spaces. 
For $a,b\in X$ (or $a,b\in Y$) the closed segment $[a,b]$ and the open 
segment $]a,b[$ connecting the points $a$ and $b$ are defined by
$$
  [a,b]:=\{ta+(1-t)b \mid 0\le t\le 1\},\qquad
  ]a,b[:=\{ta+(1-t)b \mid 0<t<1\}.
$$
(Observe that, if the endpoints $a$ and $b$ coincide then 
$[a,b]=]a,b[=\{a\}$.)

Given a convex subset $D\subset X$ and a map $f:D\to Y$, 
we can consider two {\it convexity preserving properties} 
for $f$. We say that $f$ preserves convexity if $f(K)$ is convex 
for all convex subset $K\subseteq D$. Analogously, we say that 
$f^{-1}$ preserves convexity or $f$ is inversely convexity preserving
if $f^{-1}(K)$ is convex whenever $K$ is a convex subset of $f(D)$. 
It is immediate to see that $f$ is convexity preserving if and only if
\Eq{1}{
  [f(x),f(y)]\subseteq f([x,y]) \qquad (x,y\in D).
}
On the other hand, $f$ is inversely convexity preserving if and only if
\Eq{2}{
  [f(x),f(y)]\supseteq f([x,y]) \qquad (x,y\in D).
}
Functions enjoying both of the above properties, i.e., satisfying
\Eq{0}{
  f([x,y])=[f(x),f(y)] \qquad (x,y\in D),
}
are called {\it segment preserving maps}. Therefore, $f$ is segment 
preserving if and only if it is convexity and also inversely convexity
preserving.

If, in \eq{1}, \eq{2}, and \eq{0} the closed segments are replaced by
open segments then we speak about {\it strict} convexity and segment
preserving properties for $f$. Clearly, strict convexity or segment 
preserving maps are always convexity or segment preserving (in the same 
sense), the converse, however, may not be valid (see the examples below).

The obvious candidates for (strict) convexity and segment preserving 
maps are {\it affine maps}, i.e., functions of the form $f(x)=A(x)+a$, 
where $A:X\to Y$ is linear and $a$ is a constant vector.  
It is a natural question if there exist other types of segment 
preserving functions.

One can trivially see that if $X=Y=\R$, then $f:X\to Y$ is segment
preserving if and only if it is continuous and either increasing or
decreasing; furthermore, $f$ is strictly segment preserving if and only if
it is continuous and either strictly increasing, or strictly decreasing,
or constant.  Therefore, for the first sight, the class of such maps
seems to be even more complicated in the higher-dimensional setting.
However, as we shall see, if the range of $f$ is at least two
dimensional, then the description is easier: strict segment preserving 
maps, moreover strict inversely convexity preserving maps can be 
expressed as the ratio of an $Y$-valued and a real-valued affine map.

In order to prove our results, we shall apply Ceva{}'s classical
theorem known from the elementary geometry (see \cite{Cox73}):

\begin{CThm} Let $x,y,z$ be non-collinear points of a real linear
space. Let $t_x,t_y,t_z$, $s_x,s_y,s_z$ be positive numbers and
\Eq{3a}{
  p:=\frac{t_y}{t_y+s_z}y+\frac{s_z}{t_y+s_z}z,\qquad
  q:=\frac{t_z}{t_z+s_x}z+\frac{s_x}{t_z+s_x}x,\qquad
  r:=\frac{t_x}{t_x+s_y}x+\frac{s_y}{t_x+s_y}y.
}
Then the intersection of the segments $[x,p]$, $[y,q]$, and $[z,r]$
is nonempty if and only if
\Eq{3}{
  t_x\cdot t_y\cdot t_z=s_x\cdot s_y\cdot s_z.
}
\end{CThm}

Now we briefly describe the idea how to use this theorem for segment
preserving maps.  Assume that $f:D\to Y$ is segment preserving and
$f(D)$ is at least two dimensional, i.e., it is non-collinear. Then take
$x,y,z\in D$ such that $f(x),f(y),f(z)$ form a nondegenerate triangle.
If $p\in[y,z]$, $q\in[z,x]$, and $r\in[x,y]$ are of the form \eq{3a}
such that \eq{3} holds, then (by Ceva's theorem) the intersection of the
segments $[x,p]$, $[y,q]$, and $[z,r]$ is nonempty.  Denote by $s$ its
single element. Then, the segments $[f(x),f(p)]$, $[f(y),f(q)]$, and
$[f(z),f(r)]$ have the point $f(s)$ in common. We can find positive
numbers $t'_x,t'_y,t'_z$, $s'_x,s'_y,s'_z$ such that
$$
  f(p)=\frac{t'_y}{t'_y+s'_z}f(y)+\frac{s'_z}{t'_y+s'_z}f(z),\qquad
  f(q)=\frac{t'_z}{t'_z+s'_x}f(z)+\frac{s'_x}{t'_z+s'_x}f(x),
$$
and
$$
  f(r)=\frac{t'_x}{t'_x+s'_y}f(x)+\frac{s'_y}{t'_x+s'_y}f(y).
$$
Therefore (again by Ceva's theorem), we get that
$$
  t'_x\cdot t'_y\cdot t'_z=s'_x\cdot s'_y\cdot s'_z.
$$

It will turn out that the above property yields a functional equation
for $f$ which will be solved explicitely in this paper. Thus, we will
obtain a complete characterization of strict segment preserving and
strict inversely convexity preserving functions with non-collinear range.
showing that they are rational maps, i.e., they can be expressed as
the ratio of a vector valued and a positive real valued affine map.

The use of the theory of functional equations in characterizing various
geometrical transformations has a rich literature. In the papers
\cite{Acz66a}, \cite{AczBen69}, \cite{AczMck67}, \cite{Acz71c} rational maps of the
projective and complex plane are studied by means of the theory of
functional equations.

The motivation for this research came from a recent paper \cite{Mol01b}
of Moln\'ar where the characterizations of the automorphisms of the
Hilbert space effect algebras were studied.  An effect algebra $\E(H)$
is simply the operator interval $[0,I]$ of all positive operators on a
Hilbert space $H$ bounded by the identity.  They play a fundamental role
in the so-called quantum measurement theory.  Moln\'ar has arrived at
the following problem: Is every mixture preserving bijective map
$\phi:\E(H)\to\E(H)$ a mixture-automorphism?  In other words, is it true
that any bijective map $\phi:\E(H)\to\E(H)$ with the property ``$A$ is a
convex combination of $B$ and $C$ if and only if $\phi(A)$ is also a
convex combination of $\phi(B)$ and $\phi(C)$'' is necessarily affine?
The results of this paper have significantly helped Moln\'ar to obtain
an affirmative answer to the problem described above.

The paper 

\section{Main Results}

In our first result, we describe a large class of strict segment
preserving maps.

\begin{Thm} Let $A:X\to Y$ be a linear operator, $B:X\to\R$ be a
linear function, $a\in Y$ be a vector, and $b\in\R$ be a scalar.
Let $D=\{x\mid B(x)+b>0\}$ and assume that $f:D\to Y$ is given by
\Eq{5}{
   f(x)=\frac{A(x)+a}{B(x)+b} \qquad (x\in D).
}
Then $f$ is a strict segment preserving function.
\end{Thm}

\begin{proof} Let $x,y\in D$ be fixed and let $0<t<1$. Then
\begin{eqnarray*}
  f(tx+(1-t)y)\hspace{-3mm}
   &=& \hspace{-2mm}
       \frac{A(tx+(1-t)y)+a}{B(tx+(1-t)y)+b} \\
   &=& \hspace{-2mm}
       \frac{t(A(x)+a)+(1-t)(A(y)+a)}{t(B(x)+b)+(1-t)(B(y)+b)} \\
   &=& \hspace{-2mm}
       \frac{t(B(x)+b)}{t(B(x)+b)+(1-t)(B(y)+b)}f(x)
       +\frac{(1-t)(B(y)+b)}{t(B(x)+b)+(1-t)(B(y)+b)}f(y) \\
   &\in & \hspace{-2mm} ]f(x),f(y)[.
\end{eqnarray*}
Therefore, we have that $f(]x,y[)\subseteq ]f(x),f(y)[$ for all
$x,y\in D$.

To obtain the reversed implication, observe that the function
$$
  t\mapsto \lambda(t):=\frac{t(B(x)+b)}{t(B(x)+b)+(1-t)(B(y)+b)}
$$
maps the open unit interval $]0,1[$ onto itself, therefore, for each
$s\in]0,1[$ there exists $t\in]0,1[$ such that $\lambda(t)=s$. Then,
due to the previous relations
$$
  sf(x)+(1-s)f(y)=f(tx+(1-t)y),
$$
hence, $sf(x)+(1-s)f(y)\in f(]x,y[)$ for all $s\in]0,1[$. Thus,
we get that $]f(x),f(y)[\subseteq f(]x,y[)$ for all $x,y\in D$.
\end{proof}

The main result of this paper is contained in the next theorem.  It
shows that under the non-collinear range assumption, all strict inversely
convexity preserving functions can be expressed as the ratio of two affine
functions, and as a consequence of the previous theorem, they are also
strictly segment preserving.

\begin{Thm} Let $D\subseteq X$ be a nonempty convex subset and $f:D\to Y$
be a strict inversely convexity preserving function such that $f(D)$ is
non-collinear, i.e., it contains a nondegenerate triangle. Then there exist
a linear operator $A:X\to Y$, a linear function $B:X\to\R$, a vector
$a\in Y$, and a scalar $b\in\R$ such that
\Eq{6}{
   B(x)+b>0 \qquad\mbox{for}\quad x\in D
}
and $f$ admits the representation \eq{5}.
\end{Thm}

\begin{proof} By the assumption of the theorem, $f:D\to Y$ satisfies 
\Eq{7}{
  f(]x,y[) \subseteq ]f(x),f(y)[ \qquad (x,y\in D).
}
In other words, to each point of the segment $]x,y[$, there
corresponds a unique element of $]f(x),f(y)[$.
Thus, for $x,y\in D$ with $f(x)\neq f(y)$, we can uniquely
define a function $\alpha_{x,y}:\R_+\to\R_+$ by the formula
$$
  f\Bigl(\frac{1}{1+t}x+\frac{t}{1+t}y\Bigr)
    =\frac{1}{1+\alpha_{x,y}(t)}f(x)
     +\frac{\alpha_{x,y}(t)}{1+\alpha_{x,y}(t)}f(y) \qquad (t>0).
$$

We split the rest of the proof of Theorem 2 into a sequence of lemmas.
First we derive elementary properties of the function $\alpha_{x,y}$.

\begin{Lem} For all $x,y\in D$ such that $f(x)\neq f(y)$, the function
$\alpha_{x,y}$ is strictly increasing and the following identity holds
\Eq{9}{
   \alpha_{x,y}(t)=\frac{1}{\alpha_{y,x}(1/t)} \qquad (t>0).
}
\end{Lem}

\begin{proof}[Proof of Lemma 1] Let $0<s<t$ be fixed. Then
$$
  \frac{1}{1+t}x+\frac{t}{1+t}y
  =\frac{1+s}{1+t}\Bigl(\frac{1}{1+s}x+\frac{s}{1+s}y\Bigr)
     +\frac{t-s}{1+t}y
  \in\Bigl]\frac{1}{1+s}x+\frac{s}{1+s},y\Bigr[.
$$
Therefore, by the property \eq{7} and the definition of $\alpha_{x,y}$,
\begin{eqnarray*}
  \frac{1}{1+\alpha_{x,y}(t)}f(x)
         +\frac{\alpha_{x,y}(t)}{1+\alpha_{x,y}(t)}f(y)
  & = & f\Bigl(\frac{1}{1+t}x+\frac{t}{1+t}y\Bigr) \\
  &\in& f\Bigl(\Bigl]\frac{1}{1+s}x+\frac{s}{1+s},y\Bigr[\Bigr) \\
  &\subseteq& \Bigl]\frac{1}{1+\alpha_{x,y}(s)}f(x)
            +\frac{\alpha_{x,y}(s)}{1+\alpha_{x,y}(s)}f(y),f(y)\Bigr[.
\end{eqnarray*}
Thus, there exist $0<r<1$ such that
$$
  \frac{1}{1+\alpha_{x,y}(t)}f(x)
    +\frac{\alpha_{x,y}(t)}{1+\alpha_{x,y}(t)}f(y)
  = r\Bigl(\frac{1}{1+\alpha_{x,y}(s)}f(x)
        +\frac{\alpha_{x,y}(s)}{1+\alpha_{x,y}(s)}f(y)\Bigr)+(1-r)f(y).
$$
The points $f(x)$ and $f(y)$ being different, the coefficient of $f(x)$
on both sides has to be the same. Hence
$$
  \frac{1+\alpha_{x,y}(s)}{1+\alpha_{x,y}(t)} = r<1,
$$
which yields $\alpha_{x,y}(s)<\alpha_{x,y}(t)$ proving that
$\alpha_{x,y}$ is strictly increasing.

To prove \eq{9}, let $t>0$. Then we have
\begin{eqnarray*}
 \frac{1}{1+\alpha_{x,y}(t)}f(x)
  +\frac{\alpha_{x,y}(t)}{1+\alpha_{x,y}(t)}f(y)
  &=& f\Bigl(\frac{1}{1+t}x+\frac{t}{1+t}y\Bigr) \\
  &=& f\Bigl(\frac{1}{1+1/t}y+\frac{1/t}{1+1/t}x\Bigr) \\
  &=&  \frac{1}{1+\alpha_{y,x}(1/t)}f(y)
      +\frac{\alpha_{y,x}(1/t)}{1+\alpha_{y,x}(1/t)}f(x).
\end{eqnarray*}
The coefficients of $f(x)$ are the same on the left and
right hand sides, therefore,
$$
 \frac{1}{1+\alpha_{x,y}(t)}
   =\frac{\alpha_{y,x}(1/t)}{1+\alpha_{y,x}(1/t)},
$$
which results \eq{9}.
\end{proof}

\begin{Lem} For all $x,y,z\in D$ with $f(x)\neq f(y)\neq
f(z)\neq f(x)$, the following functional equation holds
\Eq{10}{
   \alpha_{x,y}(st)=\alpha_{x,z}(s)\cdot\alpha_{z,y}(t)
   \qquad (s,t>0).
}
\end{Lem}

\begin{proof}[Proof of Lemma 2]
First assume that $x,y,z$ are such that $f(x),f(y),f(z)$ form a
nondegenerate triangle. Then $x,y,z$ also form a nondegenerate triangle.
(Indeed, if $x,y$, and $z$ were collinear, then one of them would be 
between the two others, say $z\in]x,y[$. Then, by \eq{7}, 
$f(z)\in]f(x),f(y)[$ which is an obvious contradiction.) 
Let $s,t>0$ be fixed and consider the points
$$
  p(t):=\frac{t}{1+t}y+\frac{1}{1+t}z,\qquad
  q(s):=\frac{s}{1+s}z+\frac{1}{1+s}x,\qquad
  r(st):=\frac{1}{1+st}x+\frac{st}{1+st}y.
$$
Then, by Ceva{}'s Theorem, the segments
$$
  ]x,p(t)[,\qquad
  ]y,q(s)[,\qquad\mbox{and}\qquad
  ]z,r(st)[
$$
have a nonempty intersection. Therefore, by the inverse convexity 
preserving property \eq{7} of $f$, the intersection of their images
\Eq{10A}{
  ]f(x),f(p(t))[,\qquad
  ]f(y),f(q(s))[,\qquad\mbox{and}\qquad
  ]f(z),f(r(st))[
}
is also nonempty. On the other hand, due to the definition of the
function $\alpha$, we get
$$
  f(p(t))=\frac{\alpha_{z,y}(t)}{1+\alpha_{z,y}(t)}f(y)
            +\frac{1}{1+\alpha_{z,y}(t)}f(z), \qquad
  f(q(s))=\frac{\alpha_{x,z}(s)}{1+\alpha_{x,z}(s)}f(z)
            +\frac{1}{1+\alpha_{x,z}(s)}f(x),
$$
and
$$
  f(r(st))=\frac{1}{1+\alpha_{x,y}(st)}f(x)
      +\frac{\alpha_{x,y}(st)}{1+\alpha_{x,y}(st)}f(y).
$$
Thus, again by Ceva{}'s Theorem, the non-emptiness of the intersection
of the segments \eq{10A} yields that \eq{10} is satisfied.

Finally we discuss the case when $f(x)$, $f(y)$, and $f(z)$ are pairwise
distinct collinear points. Then there exists a point $u\in D$ such that
$f(x)$, $f(y)$, $f(z)$, and $f(u)$ are non-collinear (otherwise
the range of $f$ were covered by a line). Then, using repeatedly what 
we have proved for the non-collinear case,
$$
  \alpha_{x,y}(st)
   = \alpha_{x,u}(s)\cdot \alpha_{u,y}(t)
   = \alpha_{x,z}(s)\cdot \alpha_{z,u}(1)\cdot \alpha_{u,y}(t)
   = \alpha_{x,z}(s)\cdot \alpha_{z,y}(t).
$$
Thus the proof is also complete for this case.
\end{proof}

In the next lemma we solve the functional equation \eq{10}.

\begin{Lem} There exists a positive constant $c>0$ and a
positive-valued function $\phi:D\to\R_+$ such that
\Eq{11}{
   \alpha_{x,y}(t)=\frac{\phi(y)}{\phi(x)}t^c
   \qquad\mbox{for}\quad t>0,\, x,y\in D \mbox{ with } f(x)\neq f(y).
}
\end{Lem}

\begin{proof}[Proof of Lemma 3]
Let $x_0,y_0,z_0\in D$ be fixed points such that $f(x_0),f(y_0),f(z_0)$
is non-collinear. Then, by Lemma 3, we have that the functions strictly
increasing functions $\alpha_{x_0,y_0}$, $\alpha_{x_0,z_0}$, and
$\alpha_{z_0,y_0}$ satisfy a Pexider-type functional equation. Thus,
by the theory of functional equations (see e.g.\ \cite{Acz66},
\cite{Kuc85}), there exists a constant $c\in\R$ such that
\Eq{12}{
  \alpha_{x_0,y_0}(t)=\alpha_{x_0,y_0}(1)t^c,\quad
  \alpha_{x_0,z_0}(t)=\alpha_{x_0,z_0}(1)t^c,\quad
  \alpha_{z_0,y_0}(t)=\alpha_{z_0,y_0}(1)t^c \qquad\mbox{for}\quad t>0.
}
Due to the monotonicity properties of these functions, we have that
$c>0$.

We are going to show that
\Eq{13}{
  \alpha_{x,y}(t)=\alpha_{x,y}(1)t^c
     \qquad\mbox{for}\quad t>0,\,x,y\in D\mbox{ with } f(x)\neq f(y).
}
By \eq{12}, and by Lemma 1, we have that \eq{13} holds if $x,y\in
D_0:=\{x_0,y_0,z_0\}$.

Now we verify \eq{13} if $x\not\in D_0$, $y\in D_0$ and 
$f(x)\neq f(y)$. Assume that $y=x_0$. Then one of the triples
$f(x),f(x_0),f(y_0)$ and $f(x),f(x_0),f(z_0)$ is non-collinear, 
say the first. Then, using \eq{10} and \eq{12},
$$
  \alpha_{x,x_0}(t)
     =\alpha_{x,y_0}(1)\cdot\alpha_{y_0,x_0}(t)
     =\alpha_{x,y_0}(1)\cdot\alpha_{y_0,x_0}(1)t^c
     =\alpha_{x,x_0}(1)t^c
  \qquad\mbox{for}\quad t>0.
$$
We can also get \eq{13} in the cases $y=y_0$, $y=z_0$ similarly.

The proof of \eq{13} for the case when $x\in D_0$, $y\not\in D_0$
follows from the previous case and from the identity \eq{9}.

Now we consider the case when $x,y\not\in D_0$. Then for one of the
points of $D_0$, say for $x_0$, the point $f(x_0)$ is non-collinear
together with $f(x)$ and $f(y)$. Then, using Lemma 2 again and what 
we have already shown,
$$
  \alpha_{x,y}(t)
     =\alpha_{x,x_0}(1)\cdot\alpha_{x_0,y}(t)
     =\alpha_{x,x_0}(1)\cdot\alpha_{x_0,y}(1)t^c
     =\alpha_{x,y}(1)t^c
  \qquad\mbox{for}\quad t>0.
$$
Thus \eq{13} is proved.

To complete the proof of the lemma, define $\phi:D\to\R_+$ by
$$
  \phi(x):=\left\{\begin{array}{ll}
             \alpha_{x_0,x}(1)
                     \qquad & \mbox{if } f(x)\neq f(x_0). \\[2mm]
             \alpha_{x_0,y_0}(1) \cdot \alpha_{y_0,x}(1)
                     \qquad & \mbox{if } f(x)=f(x_0).
           \end{array}\right.
$$
In view of \eq{13}, it is sufficient to prove that
\Eq{13a}{
   \alpha_{x,y}(1)=\frac{\phi(y)}{\phi(x)} \qquad\mbox{for}\quad x,y\in D
   \mbox{ with } f(x)\neq f(y).
}

We distinguish three different cases.

The first case is when $\{f(x),f(y)\}$ does not contain $f(x_0)$,
that is, when $f(x)\neq f(x_0)\neq f(y)$. Then
$$
  \alpha_{x,y}(1)
   =\alpha_{x,x_0}(1)\cdot\alpha_{x_0,y}(1)
   =\frac{\alpha_{x_0,y}(1)}{\alpha_{x_0,x}(1)}
   =\frac{\phi(y)}{\phi(x)}.
$$

The second case is when $\{f(x),f(y)\}$ contains $f(x_0)$, but it
 does not contain $f(y_0)$. Then there are two possibilities.
If $f(x)=f(x_0)$ and $f(y)\neq f(y_0)$, then
$$
  \alpha_{x,y}(1)
   =\alpha_{x,y_0}(1)\cdot\alpha_{y_0,y}(1)
   =\alpha_{x,y_0}(1)\cdot\alpha_{y_0,x_0}(1)\cdot\alpha_{x_0,y}(1)
   =\frac{\alpha_{x_0,y}(1)}{\alpha_{x_0,y_0}(1)\cdot\alpha_{y_0,x}(1)}
   =\frac{\phi(y)}{\phi(x)}.
$$
If $f(y)=f(x_0)$ and $f(x)\neq f(y_0)$, then the proof is similar
to the previous argument.

The third case is when $\{f(x),f(y)\}$ contains $f(x_0)$ and also
$f(y_0)$. Then again, there are two possibilities.
If $f(x)=f(x_0)$ and $f(y)=f(y_0)$, then
\begin{eqnarray*}
  \alpha_{x,y}(1)
   &=&\alpha_{x,z_0}(1)\cdot\alpha_{z_0,y}(1)
   =\alpha_{x,y_0}(1)\cdot\alpha_{y_0,z_0}(1)
      \cdot\alpha_{z_0,x_0}(1)\cdot\alpha_{x_0,y}(1)\\
   &=&\alpha_{x,y_0}(1)\cdot\alpha_{y_0,x_0}(1)\cdot\alpha_{x_0,y}(1)
   =\frac{\alpha_{x_0,y}(1)}{\alpha_{x_0,y_0}(1)\cdot\alpha_{y_0,x}(1)}
   =\frac{\phi(y)}{\phi(x)}.
\end{eqnarray*}
Finally, if $f(y)=f(x_0)$ and $f(x)=f(y_0)$, then we can argue in a
similar way.

Thus, in all cases, the proof of \eq{13a} is complete.
\end{proof}

\begin{Lem} The functions $f$ and $\phi$ satisfy the functional
equation
\Eq{14}{
   f\Bigl(\frac{tx+sy}{t+s}\Bigr)
   =\frac{t^c\phif(x)+s^c\phif(y)}{t^c\phi(x)+s^c\phi(y)}
   \qquad (x,y\in D,\ t,s>0).
}
\end{Lem}

Here and in the subsequent lemmas, we denote by $\phif$ the function 
defined by $\phif(x)=\phi(x)f(x)$  ($x\in D$).

\begin{proof}[Proof of Lemma 4]
If $f(x)=f(y)$, then $\dfrac{tx+sy}{t+s}\in]x,y[$ and \eq{7} yield that
$$
  f\Bigl(\frac{tx+sy}{t+s}\Bigr) \in ]f(x),f(y)[ = \{f(x)\},
$$
that is, the left hand side of \eq{14} equals $f(x)$. However, the right
hand side is also identical to $f(x)$, hence \eq{14} holds trivially
in this case.

Therefore, we can assume that $f(x)\neq f(y)$. Then, using the
definition of $\alpha_{x,y}$ and the previous lemma, we get
\begin{eqnarray*}
   f\Bigl(\frac{tx+sy}{t+s}\Bigr)
    &=&f\Bigl(\frac{1}{1+s/t}x+\frac{s/t}{1+s/t}y\Bigr)\\
    &=&\frac{1}{1+\alpha_{x,y}(s/t)}f(x)
           +\frac{\alpha_{x,y}(s/t)}{1+\alpha_{x,y}(s/t)}f(y) \\
    &=&\frac{1}{1+{(\phi(y)s^c)}/{(\phi(x)t^c})}f(x)
           +\frac{{(\phi(y)s^c)}/{(\phi(x)t^c})}
                  {1+{(\phi(y)s^c)}/{(\phi(x)t^c)}}f(y) \\
    &=&\frac{t^c\phif(x)+s^c\phif(y)}
           {t^c\phi(x)+s^c\phi(y)}.
\end{eqnarray*}
Thus \eq{14} is valid in this case as well.
\end{proof}

In the next lemma, we extend the functional equation \eq{14} to three
variables. 

\begin{Lem} For $x,y,z\in D$ such that $f(x)$, $f(y)$, and $f(z)$ are
non-collinear, the functions $f$ and $\phi$ satisfy the functional equation
\Eq{15}{
   f\Bigl(\frac{tx+sy+rz}{t+s+r}\Bigr)
   =\frac{t^c\phif(x)+s^c\phif(y)+r^c\phif(z)}
       {t^c\phi(x)+s^c\phi(y)+r^c\phi(z)}
   \qquad (t,s,r>0).
}
\end{Lem}

\begin{proof}[Proof of Lemma 5]
For $t,s,r>0$, consider the point
$$
   p(t,s,r):=\frac{tx+sy+rz}{t+s+r}.
$$
Then
$$
  p(t,s,r)
    =\frac{t}{t+s+r}x+\frac{s+r}{t+s+r}\cdot\frac{sy+rz}{s+r}
    \in \Bigl]x,\frac{sy+rz}{s+r}\Bigr[.
$$
Therefore, by the inverse convexity preserving property \eq{7} 
and by Lemma 4, we get
$$
 f(p(t,s,r))
    \in f\Bigl(\Bigl]x,\frac{sy+rz}{s+r}\Bigr[\Bigr)
    \subseteq \Bigl]f(x),f\Bigl(\frac{sy+rz}{s+r}\Bigr)\Bigr[
    = \Bigl]f(x),\frac{s^c\phif(y)+r^c\phif(z)}
                         {s^c\phi(y)+r^c\phi(z)}\Bigr[.
$$
On the other hand, we also have that
$$
   \frac{t^c\phif(x)+s^c\phif(y)+r^c\phif(z)}
       {t^c\phi(x)+s^c\phi(y)+r^c\phi(z)}
    \in \Bigl]f(x),\frac{s^c\phif(y)+r^c\phif(z)}
                         {s^c\phi(y)+r^c\phi(z)}\Bigr[.
$$
For similar reasons,
$$
  f(p(t,s,r)) \qquad\mbox{and}\qquad
  \frac{t^c\phif(x)+s^c\phif(y)+r^c\phif(z)}
       {t^c\phi(x)+s^c\phi(y)+r^c\phi(z)}
$$
also belong to the two segments
$$
    \Bigl]f(y),\frac{t^c\phif(x)+r^c\phif(z)}
                    {t^c\phi(x)+r^c\phi(z)}\Bigr[
    \qquad\mbox{and}\qquad
    \Bigl]f(z),\frac{t^c\phif(x)+s^c\phif(y)}
                    {t^c\phi(x)+s^c\phi(y)}\Bigr[.
$$
Due to the non-collinearity of $f(x)$, $f(y)$, and $f(z)$, the
intersection of the above three segments is a singleton, whence
we get that
$$
  f(p(t,s,r)) = \frac{t^c\phif(x)+s^c\phif(y)+r^c\phif(z)}
                     {t^c\phi(x)+s^c\phi(y)+r^c\phi(z)},
$$
which is exactly the desired equality \eq{15}.
\end{proof}

Based on the previous two lemmas, we derive functional equations for $\phi$
and also for $\phif$ showing that these functions are affine.

\begin{Lem} The functions $\phif$ and $\phi$ satisfy the
functional equations
\Eq{16}{
   \left\{\quad \begin{array}{rcl}
      \phif(tx+(1-t)y)&=&t\phif(x)+(1-t)\phif(y) \\
       \phi(tx+(1-t)y)&=&t\phi(x)+(1-t)\phi(y)
   \end{array}\right.
   \qquad\mbox{for}\quad x,y\in D,\,t\in[0,1]
}
and $c=1$.
\end{Lem}

\begin{proof}[Proof of Lemma 6]
First assume that $x,y\in D$ such that $f(x)\neq f(y)$. Take $z\in D$
arbitrarily such that $f(x)$, $f(y)$, and $f(z)$ are non-collinear.
Then, by Lemma 5, we have that
$$
   f\Bigl(\frac{tx+(1-t)y+rz}{1+r}\Bigr)
   =\frac{t^c\phif(x)+(1-t)^c\phif(y)+r^c\phif(z)}
       {t^c\phi(x)+(1-t)^c\phi(y)+r^c\phi(z)}
   \qquad\mbox{for}\quad t\in]0,1[,\,r>0.
$$
On the other hand, by Lemma 4, we obtain that
$$
   f\Bigl(\frac{tx+(1-t)y+rz}{1+r}\Bigr)
   =f\Bigl(\frac{1\cdot(tx+(1-t)y)+rz}{1+r}\Bigr)
   =\frac{\phif(tx+(1-t)y)+r^c\phif(z)}{\phi(tx+(1-t)y)+r^c\phi(z)}.
$$
Thus, for $t\in]0,1[$, $r>0$,
$$
  \frac{t^c\phif(x)+(1-t)^c\phif(y)+r^c\phif(z)}
       {t^c\phi(x)+(1-t)^c\phi(y)+r^c\phi(z)}
   =\frac{\phif(tx+(1-t)y)+r^c\phif(z)}{\phi(tx+(1-t)y)+r^c\phi(z)}.
$$
Rearranging this equation and using the identity
$$
  [t^c\phif(x)+(1-t)^c\phif(y)]\phi(tx+(1-t)y)
  =[t^c\phi(x)+(1-t)^c\phi(y)]\phif(tx+(1-t)y)
$$
(which is a consequence of Lemma 4), then simplifying by
$r^c\phi(z)>0$, we get that
\begin{eqnarray*}
  && t^c\phif(x)+(1-t)^c\phif(y)
             +f(z)\phi(tx+(1-t)y) \\
  && \qquad\qquad=[t^c\phi(x)+(1-t)^c\phi(y)]f(z)
             +\phif(tx+(1-t)y) \qquad (t\in[0,1]).
\end{eqnarray*}
The point $z$ being arbitrary (such that the non-collinearity
condition holds), this equality results
\Eq{17}{
   \left\{\quad \begin{array}{rcl}
      \phif(tx+(1-t)y)&=&t^c\phif(x)+(1-t)^c\phif(y) \\
       \phi(tx+(1-t)y)&=&t^c\phi(x)+(1-t)^c\phi(y)
   \end{array}\right.
   \qquad (t\in[0,1]).
}

Now we consider the case $f(x)=f(y)$. We are going to show that
\eq{17} is satisfied in this case, too.

Choose $z\in D$ such that $f(z)\neq f(x)$. Then, by \eq{17}, we
have that
$$
  \phi\Bigl(\frac{x+z}{2}\Bigr)=\frac{\phi(x)+\phi(z)}{2^c},
$$
whence
\Eq{18}{
  \phi(x)=2^c\phi\Bigl(\frac{x+z}{2}\Bigr)-\phi(z).
}
Observe that
$$
  f\Bigl(\frac{x+z}{2}\Bigr)\neq f(y)
  \qquad\mbox{and}\qquad
  f(z)\neq f(tx+(1-t)y)=f(x).
$$
Thus, using \eq{18} and applying \eq{17} twice (in the cases when it
has already been proved), we get
\begin{eqnarray*}
  t^c\phi(x)+(1-t)^c\phi(y)
    &=& (2t)^c\phi\Bigl(\frac{x+z}{2}\Bigr)+(1-t)^c\phi(y)-t^c\phi(z) \\
    &=& (1+t)^c\Bigl[\Bigl(\frac{2t}{1+t}\Bigr)^c
             \phi\Bigl(\frac{x+y}{2}\Bigr)
             +\Bigl(\frac{1-t}{1+t}\Bigr)^c\phi(y)\Bigr]
             -t^c\phi(z) \\
    &=& (1+t)^c \phi\Bigl(\frac{2t}{1+t}\cdot\frac{x+z}{2}
             +\frac{1-t}{1+t}\cdot y\Bigr)-t^c\phi(z) \\
    &=& (1+t)^c \phi\Bigl(\frac{t}{1+t}\cdot z
             +\frac{1}{1+t}\cdot (tx+(1-t)y)\Bigr)-t^c\phi(z) \\
    &=& (1+t)^c\Bigl[\Bigl(\frac{t}{1+t}\Bigr)^c \phi(z)
             +\Bigl(\frac{1}{1+t}\Bigr)^c\phi(tx+(1-t)y)\Bigr]
             -t^c\phi(z) \\
    &=& \phi(tx+(1-t)y).
\end{eqnarray*}
Repeating the same argument with the function $\phif$, we get that
the first equation of \eq{17} is also valid for all $x,y\in D$.

Putting $x=y$ into the second equation of \eq{17}, we get that
$t^c+(1-t)^c=1$ for all $t\in[0,1]$. Hence, $c=1$, and then \eq{17}
reduces to \eq{16} which was to be proved.
\end{proof}

\begin{Lem} Then there exist a linear operator $A:X\to Y$, a linear
function $B:X\to\R$, a vector $a\in Y$, and a scalar $b\in\R$ such
that \eq{6} is valid and
\Eq{19}{
   \left\{\quad \begin{array}{rcl}
      \phif(x)&=&A(x)+a \\
       \phi(x)&=&B(x)+b
   \end{array}\right.
   \qquad (x\in D).
}
\end{Lem}

\begin{proof}[Proof of Lemma 7]
We prove the statement for the function $\phi$. The proof for $\phif$
is identical. We know from Lemma 6, that $\phi$ satisfies the second
equation of \eq{16}. Let $0<t<1$ be fixed. Then the meaning of \eq{16}
is that $\phi$ is $t$-affine on $D$. By the extension theorem of
\cite[Theorem 5]{Pal02b}, there exists a uniquely determined extension
$\phi_t:\aff(D)\to\R$ of $\phi$ which satisfies the functional equation
$$
  \phi_t(tx+(1-t)y)= t\phi_t(x)+(1-t)\phi_t(y) 
    \qquad (x,y\in\aff(D)),
$$
where
$$
  \aff(D):=\{sx+(1-s)y\mid x,y\in D,\,s\in\R\},
$$
which is called the {\it affine hull of $D$}. One can check that,
if $D$ is convex, then $\aff(D)$ is the smallest affine subspace
containing $D$. On the other hand, by the results of \cite{DarPal87},
a $t$-affine function is always $(1/2)$-affine, or in other terms,
Jensen-affine. Therefore, $\phi_t$ also satisfies
$$
  \phi_t\Bigl(\frac{x+y}{2}\Bigr)= \frac{\phi_t(x)+\phi_t(y)}{2}
  \qquad (x,y\in\aff(D)),
$$
that is, $\phi_t$ is a Jensen-affine extension of $\phi$ to $\aff(D)$.
Due to the uniqueness of such extensions, we get that $\phi_t=\phi_{1/2}$
for all $t\in]0,1[$. Consequently, $\phi^*:=\phi_{1/2}$ satisfies
\Eq{20}{
  \phi^*(tx+(1-t)y)= t\phi^*(x)+(1-t)\phi^*(y) 
    \qquad (x,y\in\aff(D)).
}

Let $x_0$ be an arbitrary point of $D$. Then $L:=\aff(D-x_0)$
is a linear subspace of $X$. It follows form the results
concerning Jensen-affine functions that there exists an additive
function $B^*:L\to\R$ and a constant $b^*\in\R$ such that
\Eq{21}{
  \phi^*(x)=B^*(x-x_0)+b^* \qquad (x\in \aff(D)).
}
Putting this form of $\phi^*$ into \eq{20}, we get (with the
substitutions $u=x-x_0$ and $v=y-x_0$) that
$$
  B^*(tu+(1-t)v)=tB^*(u)+(1-t)B^*(v) 
   \qquad (u,v\in L, \, t\in[0,1]).
$$
Using the additivity of $B^*$ this reduces (with $w=u-v$) to
$$
  B^*(tw)=tB^*(w) \qquad (w\in L,\, t\in[0,1]).
$$
Thus $B^*$ is also homogeneous, and hence it is linear. Now,
using transfinite induction, we can construct a linear extension
$B:X\to\R$ of $B^*$. Then, it follows from \eq{21} that
$$
  \phi(x)=\phi^*(x)=B^*(x-x_0)+b^*=B(x-x_0)+b^*
         =B(x)+b^*-B(x_0)=B(x)+b
  \qquad (x\in D),
$$
i.e., $\phi$ is of the desired form \eq{19}.
\end{proof}

The proof of Theorem 2 is now completed, since, by Lemma 7,
we have that
$$
  f(x)=\frac{\phif(x)}{\phi(x)}=\frac{A(x)+a}{B(x)+b}
  \qquad (x\in D).
$$
\end{proof}

Taking $D=X$ in our main theorem, we get the following consequence.

\begin{Cor} Let $f:X\to Y$ be a strict inversely convexity preserving 
function such that $f(X)$ is non-collinear.  Then there exist a linear
operator $A:X\to Y$ and a vector $a\in Y$ such that
\Eq{22}{
   f(x)=A(x)+a \qquad (x\in X).
}
\end{Cor}

\begin{proof} By Theorem 2, $f$ is of the form \eq{5} and \eq{6} is
satisfied on $D=X$. Therefore, $B$ must be identically zero. Thus,
$b>0$. Without loss of generality, we can take $b=1$. Then \eq{5}
reduces to the desired representation \eq{22}.
\end{proof}


\end{document}